\documentclass[11pt]{article} 
\usepackage[utf8]{inputenc}
\usepackage[babel]{microtype}
\usepackage[english]{babel}

\usepackage{amsmath,amssymb,amsthm, mathrsfs, mathtools,bbm,bm}
\usepackage{float}
\usepackage[dvipsnames]{xcolor}
\usepackage{multirow}
\usepackage{enumerate}
\usepackage{enumitem}
\setlist[enumerate,1]{label={(\roman*)}}
\setlist[enumerate,2]{label={(\alph*)}}
\setlist[enumerate,3]{label={(\arabic*)}}
\setlist[itemize]{nolistsep,noitemsep, topsep=0pt}
\usepackage[bmargin=20mm,tmargin=20mm,lmargin=20mm,rmargin=20mm]{geometry}
\usepackage[hypertexnames=false]{hyperref}
\hypersetup{
    colorlinks,
    linkcolor={red!60!black},
    citecolor={green!60!black},
    urlcolor={blue!60!black}
}

\usepackage[capitalize]{cleveref}
\crefname{equation}{}{}

\theoremstyle{plain}
\newtheorem{theorem}{Theorem}[section]

\newtheorem{proposition}[theorem]{Proposition}
\newtheorem{lemma}[theorem]{Lemma}

\newtheorem{problem}[theorem]{Problem}
\theoremstyle{definition}

\newtheorem{observation}[theorem]{Observation}

\usepackage{graphicx}

\graphicspath{{images/}} 
\input{tikz}

\allowdisplaybreaks

\begin{document}

\author{
Shagnik Das\,\thanks{Department of Mathematics, National Taiwan University, Taipei 10617, Taiwan. Research supported by Taiwan NSTC grant 113-2628-M-002-008-MY4.}
\thanks{E-mail: {\tt shagnik@ntu.edu.tw}}
\and
Ying-Sian Wu\,\footnotemark[1]
\thanks{E-mail: {\tt puggywu15@gmail.com}}
}

\sloppy

\title{New results on the odd- and unique-Ramsey numbers}

\maketitle

\begin{abstract}
The odd-Ramsey number $r_{\text{odd}}(n,H)$ of a graph $H$ is the minimum number of colours needed to edge-colour $K_n$ so that in every copy of $H$ some colour occurs an odd number of times, and the unique-Ramsey number  $r_{\text{u}}(n,H)$ is the corresponding notion in which some colour is required to occur not only an odd number of times but exactly once.
    
In this paper, we address three questions from previous papers. We show $r_{\text{odd}}(n,K_{s,t})> n^{1/\left(\frac s2+\frac 1{2\lfloor t/8 \rfloor}\right)}$ when $s\leq t$ and $s$ is odd and $t$ is even, which is log-asymptotically tight when $s$ is fixed and $t\to\infty$. Next, we consider the odd-Ramsey number when the host graph to be edge-coloured is a super-Dirac graph, and show that the smallest minimum degree that forces a non-trivial odd-Ramsey number of Hamilton cycles is $n/2+1$. Finally, we show that $r_\text{u}(n,C_n)> n/4$, which leads to a polynomial gap between $r_\text{odd}(n,C_n)$ and $r_\text{u}(n,C_n)$.
\end{abstract}

\section{Introduction}

\subsection{\textbf{From the classical Ramsey problem to the odd-Ramsey number}}

\noindent The classical $r$\textit{-colour diagonal Ramsey number} $R_r(k)$ is the smallest integer $n$ such that every $r$-colouring of $E(K_n)$ contains a monochromatic copy of $K_k$. This problem reamins poorly understood even for small $r$ and $k$. For instance, when $r=2$, early works of Erdős \cite{Erds1947SomeRO} and Erdős and Szekeres \cite{erdos1935combinatorial} established the bounds $2^{k/2}\leq R_2(k)\leq 4^{k}$. After that, not until 2023 did Campos, Griffiths, Morris, and Sahasrabedhe \cite{campos2023exponential} bring the first exponential improvement on the upper bound, further optimised by Gupta, Ndiaye, Norin and Wei \cite{gupta2024optimizingcgmsupperbound}. However, the gap between the lower and upper bounds remains large. There are also recent breakthroughs on the upper bounds for the general case $r\geq 2$, given by Balister et al. \cite{balister2026upperboundsmulticolourramsey}, and the lower bounds on the closely-related \textit{off-diagonal Ramsey numbers}, given by Mattheus and Verstraete \cite{mattheus2024asymptoticsr4t}, but many questions remain open.

In the Ramsey problem, one can consider host graphs  and subgraphs other than cliques. Also, instead of fixing the number of colours, one can fix the host graph and ask for the minimum number of colours required to avoid a problematic colouring pattern in a subgraph. This motivates the \textit{generalised Ramsey number} $f(G,H,q)$ for the graph $H$ in the graph $G$, posed by Erdős and Gy\'{a}rf\'{a}s \cite{Erds1997AVO}, defined to be the minimum number of colours needed to colour $E(G)$ so that every copy of $H$ in $G$ receives at least $q$ different colours. Note that if $q=2$ and $G=K_n$, $H=K_k$, then this is in essence the original Ramsey problem. This problem has attracted great attention over the past few decades, see \cite{axenovich2000generalized,bennett2024generalized,bennett2022erd,conlon2015erdHos,eichhorn2000note,fox2009ramsey,mubayi1998edge} for instance.

Aside from counting the number of colours, one can instead impose other conditions on the colour patterns appearing on each copy of $H$. The \textit{odd-Ramsey number} $r_{\text{odd}}(G,H)$ of $H$ in $G$ is the minimum number of colours needed to colour $E(G)$ so that in every copy of $H$ some colour occurs an odd number of times. We write $r_{\text{odd}}
(n,H)$ for $r_\text{odd}(G,H)$ if $G=K_n$. Note that $r_{\text{odd}}(G,H)$ is trivially equal to $1$ if $e(H)$ is odd, but when $e(H)$ is even it is a parameter worth studying. Being a natural Ramsey Theory variant, it is interesting in its own right, and is closely related to the generalised Ramsey number by $r_{\text{odd}}(G,H)\leq f(G,H,\lfloor e(H)/2 \rfloor+1)$, because if every copy of $H$ receives at least $\lfloor e(H)/2\rfloor+1$ different colours, then by the pigeonhole principle some colour occurs exactly once, and in particular it occurs an odd number of times.

The odd-Ramsey number   was first implicitly introduced by Alon \cite{alon2024graph} in his work on \textit{graph-codes}. Given a graph $H$, an $H$-code is a family of graphs on a common set of vertices such that the symmetric difference of any two of its members is not isomorphic to $H$. More precisely, making the collection of graphs on vertex set $[n]$ into a vector space $\mathbb F_2^{K_n}$ over $\mathbb F_2$ by identifying each graph with its edge set and defining the sum of two graphs $G_1=([n],E_1)$, $G_2=([n],E_2)$ to be the graph on $[n]$ whose edge set is $E_1\Delta E_2$, Alon defined
\begin{align*}
    d_H(n)=\frac 1{2^{\binom n2}}\max \Big\{|\mathcal C|:\mathcal F\subseteq \mathbb F_2^{K_n} \text{, and there are no }G_1,G_2\in \mathcal C\text{ such that }G_1+G_2\simeq H\Big\},
\end{align*}
the maximum fraction of the number of graphs in an $H$-code, and implicitly showed that $r_{\text{odd}}(n,H)$ is closely related to $d_H(n)$ as follows. The proof was explicitly given by Versteegen \cite{versteegen2024upperboundslineargraph}.
\begin{theorem}[Alon \cite{alon2024graph}]
    Let $H$ be a graph. If $r_\mathrm{odd}(n,H)$ is sufficiently large, then
    \begin{align*}
        d_H(n)\geq r_\mathrm{odd}(n,H)^{-(e(H)/2+1)}.    
    \end{align*}
\end{theorem}
One can moreover define $d_H^{\text{lin}}(n)$ by requiring the $H$-code $\mathcal C$ in the definition of $d_H(n)$ to be linear. In this case, it turns out that $r_\text{odd}(n,H)^{-1}\geq d_H^{\text{lin}}(n)\geq r_\text{odd}(n,H)^{-(e(H)/2+1)} $, again shown implicitly by Alon \cite{alon2024graph} and explicitly by Versteegen \cite{versteegen2024upperboundslineargraph}. Therefore, any bound on $r_\text{odd}(n,H)$ or $d_H^{\text{lin}}(n)$ will give a bound on the other. These relations to graph-codes further justify the importance of the odd-Ramsey numbers.
\subsection{Recent progress on the odd-Ramsey number and its variants}

Despite being a very new Ramsey-type concept, there are already a number of results on the odd-Ramsey numbers. Alon \cite{alon2024graph} implicitly dealt with the odd-Ramsey number of stars, matchings and certain family of cliques. Later on, several groups of authors have been seeking to determine $r_\text{odd}(n,K_t)$. Cameron and Heath \cite{cameron2023new}
 proved that $r_{\text{odd}}(n,K_4)=n^{o(1)}$, and then in 2023 Ge, Xu, and Zhang \cite{ge2023newvarianterdhosgyarfasproblem}  conjectured  $r_\text{odd}(n,K_t)=n^{o(1)}$ for all $t$. Bennett, Heath, and Zerbib \cite{bennett2023edge} and, independently, Ge, Xu, and Zhang \cite{ge2023newvarianterdhosgyarfasproblem} proved the $t=5$ case. The next nontrivial case is $t=8$, proved by Yip \cite{yip2025varianterdhosgyarfasproblemk8} in 2024. As for an arbitrary graph $H$ with an even number of edges, in 2024 Versteegen \cite{versteegen2024upperboundslineargraph} proved the general lower bound $r_\text{odd}(n,H)=\Omega(\log n)$, and showed that we further have a polynomial lower bound if $H$ has an even-decomposition, namely a decomposition into independent sets in a particular way. Versteegen showed that almost all graphs with an even number of edges have even-decompositions. Closely related to this paper, in 2024 Boyadzhiyska, Das, Lesgourgues, and Petrova \cite{boyadzhiyska2024oddramseynumberscompletebipartite} discussed the odd-Ramsey number of complete bipartite graphs, obtaining results for the family of all spanning complete bipartite graphs on $n$ vertices, its subfamilies, and for individual fixed $K_{s,t}$. \linebreak In particular, they showed $r_\text{odd}(n,K_{s,t})\geq (1+o(1))(\frac nt)^{1/\lceil \frac s2\rceil}$ if $s\leq t$ and $st$ is even. On the other hand, Axenovich, F\"{u}redi, and Mubayi \cite{axenovich2000generalized} proved that $f(K_n,K_{s,t},st/2+1)\leq O(n^{\frac{2s+2t-4}{st}})$, and this gives the best current upper bound $ r_\text{odd}(n,K_{s,t})\leq O(n^{\frac{2s+2t-4}{st}})$ on the odd-Ramsey number of $K_{s,t}$ when $st$ is even.

 There are also results on $r_{\text{odd}}(G,H)$ where $G\ne K_n$. In 2023, Bennett, Heath, and Zerbib \cite{bennett2023edge} proved that $r_{\text{odd}}(K_{n,n},K_{2,2})=n/2+o(n)$, and in 2025 Crawford, Heath, Henderschedt, Schwieder, and Zerbib \cite{crawford2025oddramseynumbersmultipartite} showed that $r_{\text{odd}}(K_{n,n},K_{2,t})=n/t+o(n)$. Recently in 2025, Boyadzhiyska, Das, Lesgourgues, and Petrova \cite{boyadzhiyska2025oddramseynumbershamiltoncycles} initiated the study of $r_\text{odd}(G,H)$ where $H=C_n$ is a Hamilton cycle for even $n$ and $G$ is a \textit{Dirac graph}, namely graphs satisfying Dirac's condition $\delta(G)\geq n/2$. Dirac's condition guarantees the existence of many Hamiltonian cycles in the graph. For instance, Nash-Williams \cite{nash1971edge} proved that every Dirac graph contains linearly many edge-disjoint Hamiltonian cycles, while S\'{a}rk\"{o}zy, Selkow, and Szemer\'{e}di \cite{sarkozy2003number} showed that every $n$-vertex Dirac graph contains at least $c^nn!$ many Hamilton cycles. It is then natural to ask how degree conditions on $G$ affect the odd-Ramsey number $r_{\text{odd}}(G,H)$. To this end, for all even integer $n\geq 4$, Boyadzhiyska, Das, Lesgourgues, and Petrova \cite{boyadzhiyska2025oddramseynumbershamiltoncycles} defined what we call the \textit{sparse odd-Ramsey number} of Hamilton cycles by
 \begin{align*}
     r_\text{odd}(n,d;C_n)=\min\{ r_\text{odd}(G,C_n):v(G)=n,\text{ }\delta(G)\geq d \}.
 \end{align*}
This is a generalisation of $r_\text{odd}(n,C_n)$ since $r_\text{odd}(n,C_n)=r_\text{odd}(n,n-1;C_n)$. Note that if $0\leq d<n/2$, then there is always an $n$-vertex graph $G$ with $\delta(G)=d$ containing no Hamilton cycles, so we shall consider $r_\text{odd}(n,d;C_n)$ when $d\geq n/2$. In \cite{boyadzhiyska2025oddramseynumbershamiltoncycles} it was   shown that 
$r_{\text{odd}}(n,n/2+k;C_n)\leq \min\{2k+2,\frac{3k}{\sqrt{2n}}+\frac{3\sqrt{2n}}{4}+3\}$\linebreak for all $0\leq k\leq n/2$. As for the lower bounds, it is trivial that $r_\text{odd}(n,d;C_n)\geq 2$ for all $d\geq n/2$ because any $n$-vertex graph $G$ with $\delta(G)\geq d$ satisfies Dirac's condition and contains a Hamilton cycle. Beyond this trivial bound, closely related to this paper, in \cite{boyadzhiyska2025oddramseynumbershamiltoncycles} is was shown that $r_{\text{odd}}(n,n/2+4;C_n)\geq 3$. That is, a minimum degree  $\delta(G)\geq n/2+4$ guarantees nontrivial $r_\text{odd}(G,C_n)$.

Finally, in 2025 Yip \cite{yip2025varianterdhosgyarfasproblemk8} (also introduced by Radoicic and studied by Axenovich and Conlon in unpublished works) defined the \textit{unique Ramsey number} $r_\text{u}(G,H)$ to be the minimum number of colours needed to colour $E(G)$ so that in every copy of $H$ some colour occurs exactly once. This is worth studying because as we discussed earlier, the trivial upper bound $r_\text{odd}(G,H)\leq f(G,H,\lfloor e(H)/2\rfloor+1)$ on the odd-Ramsey number actually comes from the relation $r_\text{odd}(G,H)\leq r_\text{u}(G,H)\leq  f(G,H,\lfloor e(H)/2\rfloor+1)$, so proving or disproving a separation between $r_\text{odd}(G,H)$, $r_\text{u}(G,H)$ or $r_\text{u}(G,H)$, $f(G,H,\lfloor e(H)/2\rfloor+1)$ will be valuable. Yip \cite{yip2025varianterdhosgyarfasproblemk8}
showed that if $H$ is not a clique and has no isolated vertices, then $r_\text{u}(n,H)=n^{\Omega(1)}$, while for cliques he
conjectured $r_\text{u}(K_t)=n^{o(1)}$ for all $t$, which is claimed to be true for $t\leq 7$ by Conlon in an unpublished work. As for separation results, 
Boyadzhiyska, Das, Lesgourgues, and Petrova \cite{boyadzhiyska2024oddramseynumberscompletebipartite} showed a constant factor gap between $r_\text{odd}(n,\mathcal H)$ and $r_\text{u}(n,\mathcal H)$, where $\mathcal H$ is the collection of all cliques of size at most $n$. In \cite{boyadzhiyska2025oddramseynumbershamiltoncycles} these authors further showed $(\frac{\sqrt 2}2+o(1))\sqrt n\leq r_\text{odd}(n,C_n)\leq \frac{3\sqrt2}2\sqrt n$ and posed the question to determine $r_{\text {u}}(n,C_n)$. They showed in personal correspondence that $r_{\text {u}}(n,C_n)\leq n/2+1$, the construction given by Proposition \ref{Upper}, and they believed that
$r_\text{u}(n,C_n)$ is linear in $n$, far away from the $\Theta(\sqrt n)$ magnitude of $r_\text{odd}(n,C_n)$. 
\begin{proposition}[Boyadzhiyska, Das, Lesgourgues, and Petrova]\label{Upper}
    Let $n\geq 4$ be an even integer, then
    \begin{align*}
        r_\mathrm{u}(n,C_n)\leq n/2+1.
    \end{align*}
\end{proposition}
\begin{proof}
    We shall $(n/2+1)$-colour $K_n$ so that every Hamilton cycle uses some colour exactly once. Partition
    \begin{align*}
        V(K_n)=V_1\sqcup V_2,\quad |V_1|=n/2+1,\quad |V_2|=n/2-1.
    \end{align*}
    Let $V_1=\{v_1,\ldots,v_{n/2+1}\}$. Define the colouring $\chi:V(K_n)\to [n/2+1]$ by
    \begin{align*}
        \chi(e)=\begin{cases*}
            1,&$\text{ if }e\subseteq V_1\text{ or }e\subseteq V_2$,\\
            i,&$\text{ if }e\nsubseteq V_1\text{ and }e\nsubseteq V_2\text{ and }v_i\in e$.
        \end{cases*}
    \end{align*}
    Let $C$ be any Hamilton cycle, then $E(C)\cap E(K_n[V_1])\ne\emptyset$ since $|V_1|>|V_2|$. Let $v_{i_1}v_{i_2}\ldots v_{i_k}\subseteq C$ be a maximal subpath of $C$ contained in $V_1$, then any colour in $\{i_1,i_k\}\setminus\{1\}$ occurs exactly once in $C$.
\end{proof}

\subsection{Our results}
In \cite[Theorem 1.5]{boyadzhiyska2024oddramseynumberscompletebipartite} it was shown that $r_\text{odd}(n,K_{s,t})\geq (1+o(1))(\frac nt)^{1/\lceil \frac s2\rceil}$ if $s\leq t$ and $st$ is even. Together with the upper bound $O(n^{\frac{2s+2t-4}{st}})$, if we fix an even $s$ and let $t\to \infty$, then the bounds are log-asymptotically tight. That is, the exponents of the bounds are asymptotically tight. Our first result is an improvement on the lower bound when $s$ is odd and $t\geq 8$, showing that for odd $s$ we also have log-asymptotically tight bounds as $t\to\infty$ while keeping $st$ even.
\begin{theorem}\label{thmKST}
    Let $3\leq s\leq t$ be fixed integers, $s$ be odd and $t$ be even, then 
    \begin{align*}
        r_\mathrm{odd}(n,K_{s,t})> n^{1/(\frac s2+\frac 1{2\lfloor t/8\rfloor})}.
    \end{align*}
\end{theorem}
Next, in \cite[Theorem 1.3]{boyadzhiyska2025oddramseynumbershamiltoncycles} it was shown that while $r_\text{odd}(n,n/2;C_n) = 2$, we have $r_\text{odd}(n,n/2+4;C_n)\geq 3$. The authors suggested that the transition from $r_\text{odd}(n,n/2;C_n)=2$ to $r_\text{odd}(n,n/2+4;C_n)\geq 3$ happens already at minimum degree $n/2+1$, and our second result shows that this is indeed the case.
\begin{theorem}\label{thmDIRAC}
    Let $n\geq 6$ be an even integer, then
    \begin{align*}
        r_\mathrm{odd}(n,n/2+1;C_n)\geq 3.
    \end{align*}
\end{theorem}
Finally, in \cite[Question 5.5]{boyadzhiyska2025oddramseynumbershamiltoncycles} the authors asked for $r_\text{u}(n,C_n)$. Our last result shows that it is indeed linear in $n$ as they believed, thus demonstrating a polynomial gap between $r_\text{odd}(n,C_n)$ and $r_\text{u}(n,C_n)$. Moreover, with the upper bound in Proposition \ref{Upper}, this determines $r_\text{u}(n,C_n)$ up to a multiplicative factor of $2$.
\begin{theorem}\label{thmUNIQ}
    Let $n\geq 4$ be an integer, then 
    \begin{align*}
        r_\mathrm{u}(n,C_n)> n/4.
    \end{align*}
\end{theorem}

\textbf{Notation. } Our notation is mostly standard. Unless otherwise specified, the term colouring refers to an edge-colouring and an $r$-colouring uses the colour-palette $[r]$.  We say that a subgraph of an edge-coloured graph $G$ is \textit{odd-chromatic} if some colour occurs an odd number of times in it, and \textit{even-chromatic} otherwise.

The \textit{length} of a path is the number of edges in it. Given distinct vertices $v_1,v_2,\ldots,v_k$, we denote by $v_1v_2\ldots v_k$ the path with edge set $\{v_iv_{i+1}:i\in [k-1]\}$, and by $v_1v_2\ldots v_kv_1$ the cycle with edge-set $\{v_iv_{i+1}:i\in[k-1]\}\cup\{v_kv_1\}$. If $P:v_jv_{j+1}\ldots v_{j+\ell}$ is a path, then we write $v_1v_2\ldots v_{j-1}v_j P v_{j+\ell}v_{j+\ell+1}\ldots v_k$ for the path $v_1v_2\ldots v_k$. Given vertices $a,b$, an $\{a,b\}$\textit{-path} is a path having endpoints $a$ and $b$. A \textit{cherry} is a path of length two, a \textit{claw} is the star $K_{1,3}$, and a \textit{$2$-matching} is a matching containing exactly two edges.

\textbf{Organisation of the paper.} We prove Theorem \ref{thmKST} in Section \ref{sec2}, prove Theorem \ref{thmDIRAC} in Section \ref{sec3}, and prove Theorem \ref{thmUNIQ} in Section \ref{sec4}. In section \ref{sec5}, we summarise some related open problems and directions for further study.

\section{Odd-Ramsey numbers of complete bipartite graphs}\label{sec2}

Here we prove Theorem \ref{thmKST}, improving the bound $r_\text{odd}(n,K_{s,t})\geq (1+o(1))(\frac nt)^{1/\lceil \frac s2\rceil}$ in \cite[Theorem 1.5]{boyadzhiyska2024oddramseynumberscompletebipartite} when $s$ is odd and $t\geq 8$. We first observe that when $s$ is even, directly from the proof of \cite[Theorem 1.5]{boyadzhiyska2024oddramseynumberscompletebipartite} one can actually deduce more. Letting $s$ be even and $s\leq t$, we call a coloured copy of $K_{s,t}$ \textit{strongly-even-chromatic} if in its bipartition
\begin{align*}
    V(K_{s,t})=V_1\sqcup V_2,\quad |V_1|=s,\quad |V_2|=t,
\end{align*}
for all $u\in V_2$, in the edge set $\{uv:v\in V_1\}$ every colour occurs an even number of times. With this notation, when $s$ is even, it turns out that if we colour $K_n$ with fewer than $(1+o(1))(\frac nt)^{2/s}$ colours, then not only can we get an even-chromatic copy of $K_{s,t}$, but actually a strongly-even-chromatic one. Therefore, we restate the result of  \cite[Theorem 1.5]{boyadzhiyska2024oddramseynumberscompletebipartite} when $s$ is even as follows, and include its proof for completeness.
\begin{lemma}[Boyadzhiyska, Das, Lesgourgues, and Petrova\cite{boyadzhiyska2024oddramseynumberscompletebipartite}]\label{strongEC}
    Let $s'\leq t'$ be integers and $s'$ be even. If there is an $r$-colouring of $K_n$ under which no copy of $K_{s',t'}$ is strongly-even-chromatic, then 
    \begin{align*}
        r\geq (1+o(1))\left(\frac n{t'}\right)^{2/s'}.
    \end{align*}
\end{lemma}
\begin{proof}
    Fix an $r$-colouring of $K_n$ under which no copy of $K_{s',t'}$ is strongly-even-chromatic. For all $u\in V(K_n)$, we say an $s'$-vertex set $S\subseteq V(K_n)$ is an \textit{even} $s'$\textit{-neighborhood} of $u$ if in the edge set $\{uv:v\in S\}$ every colour occurs an even number of times. We shall double count the set
    \begin{align*}
        \mathcal E=\Big\{(u,S): S\text{ is an even }s'\text{-neighborhood of }u\Big\}.
    \end{align*}
    On one hand, for each vertex $u\in V(K_n)$ we can
    lower-bound the number of its even $s'$-neighborhoods by \cite[Lemma 4.1]{boyadzhiyska2024oddramseynumberscompletebipartite}. The idea is that an even-$s'$-neighborhood of $u$ can be built by picking $ s'/2 $ ordered pairs of monochromatic edges incident to it, and each even-$s'$-neighborhood can be counted this way in at most $s'$ different orders, so using Jensen's inequality one can get the lower bound 
    \begin{align*}
        |\{\text{even }s'\text{-neighborhoods of }u\}|\geq \frac{(n-s')^{s'/2}(n-r-s')^{s'/2}}{r^{s'/2}s'!}.
    \end{align*}
    On the other hand, any $s'$-vertex set $S$ can be the even-$s'$-neighborhood of at most $t'-1$ distinct vertices, otherwise there will be a strongly-even-chromatic copy of $K_{s',t'}$. Therefore,
    \begin{align*}
        \binom n{s'}(t'-1)\geq |\mathcal E|\geq \frac{n(n-s')^{s'/2}(n-r-s')^{s'/2}}{r^{s'/2}s'!}.
    \end{align*}
    As long as $r=o(n)$, this reduces to
    \begin{align*}
        &\frac{n^{s'}}{s'!}t'\geq (1+o(1))\frac{n^{s'+1}}{r^{s'/2}s'!},\quad r\geq (1+o(1))\left(\frac n{t'}\right)^{2/s'}.\qedhere
    \end{align*}
\end{proof}
In \cite[Theorem 1.5]{boyadzhiyska2024oddramseynumberscompletebipartite} when $s$ is odd,  the authors defined $S$ to be an even-$s$-neighborhood of $u$ if in $\{uv:v\in S\}$ exactly one colour occurs an odd number of times, and then the proof above can be modified to obtain their lower bound. The main problem is that, when upper bounding $|\mathcal E|$, now any $s$-vertex set $S$ can possibly be the even-$s$-neighborhood of up to $r+t-1$ distinct vertices while still not creating an even-chromatic copy of $K_{s,t}$, and this leads to a weaker bound on $r$.

We resolve this problem in Theorem \ref{thmKST} by first applying Lemma \ref{strongEC} with $s'=s-3$ and some large $t'$ to get a strongly-even-chromatic copy of $K_{s-3,t'}$ with bipartition $V(K_{s-3,t'})=V_1\sqcup V_2$, $|V_1|=s-3$, $|V_2|=t'$. Then it suffices to find an even-chromatic copy of $K_{3,t}$ with bipartition
\begin{align*}
    V(K_{3,t})=U_1\sqcup U_2,\quad |U_1|=3,\quad |U_2|=t,\quad U_1\cap V_1=\emptyset,\quad U_2\subseteq V_2.
\end{align*}
The latter task turns out to be equivalent to a hypergraph even-cover problem. In a hypergraph $\mathcal H=(V,\mathcal E)$, an \textit{even-cover} of size $k$ is a set $\mathcal F\subseteq \mathcal E$ of $k$ edges such that every vertex $v\in V$ is contained in an even number of members of $\mathcal F$. We need the following extremal  result of hypergraph even-covers from Naor and Verstraete \cite[Theorem 1.2]{naor2008parity}.
\begin{lemma}[Naor and Verstraete \cite{naor2008parity}]\label{evencover}
    Let $q\geq 2$ and $\mathcal H=(V,\mathcal E)$ be an $n$-vertex hypergraph with $|e|\leq q$ for all $e\in \mathcal E$. If $\mathcal H$ does not contain an even-cover of size $k$, then 
    \begin{align*}
        |\mathcal E(H)|=o(n^{\frac q2+\frac{\lceil q/3\rceil}{2\lfloor k/8 \rfloor}}).
    \end{align*}
\end{lemma}
With this, taking $t'$ large enough we can prove Theorem \ref{thmKST}.
\begin{proof}[\proofname\mbox{} of Theorem 
\ref{thmKST}]
    Let $r=n^{1/(\frac s2+\frac 1{2\lfloor t/8\rfloor})}$. To prove $r_\text{odd}(n,K_{s,t})> r$, it suffices to show that in every $r$-colouring $\chi:E(K_n)\to [r]$, there is an even-chromatic copy of $K_{s,t}$. Fix one such $\chi$. First, note that
    \begin{align*}
        r\leq \left( \frac n{r^{\frac 32+\frac 1{2\lfloor t/8\rfloor}}} \right)^{\frac 2{s-3}},
    \end{align*}
    so
    by Lemma \ref{strongEC} with $s'=s-3$ and $t'=r^{\frac 32+\frac 1{2\lfloor t/8\rfloor}}$, there is a strongly-even-monochromatic copy of $K_{s-3,t'}$. Let its bipartition be
    \begin{align*}
        V(K_{s-3,t'})=V_1\sqcup V_2,\quad |V_1|=s-3
        ,\quad |V_2|=t'.
    \end{align*}
    Now we fix any three vertices $w_1,w_2,w_3\in V(K_n)\setminus (V_1\sqcup V_2)$, and define a hypergraph $\mathcal H=(V,\mathcal E)$ with $V(\mathcal H)=[r]$ being the colour-palette and
    \begin{gather*}
        \mathcal{E}(\mathcal H)=\{e_u:u\in V_2\},\\
        e_u=
        \begin{cases*}
            \{\chi(u,w_1),\chi(u,w_2),\chi(u,w_3)\}, & if $\chi(u,w_1),\chi(u,w_2),\chi(u,w_3)$ are pairwise distinct,\\
            \{\chi(u,w_i)\}, & if $\{i,j,k\}=\{1,2,3\}$ and $\chi(u,w_j)=\chi(u,w_k)$.
        \end{cases*}
    \end{gather*}
    If there are pairwise distinct $u_1,u_1',\ldots,u_{t/2},u_{t/2}'\in V_2$ such that $e_{u_i}=e_{u_i'}$ for each $i$, then $\{w_1,w_2,w_3\}$ and $\{u_1,u_1',\ldots,u_{t/2},u_{t/2}'\}$ will form an even-chromatic copy of $K_{3,t}$ because each $u_i$ sees the same colours to $\{w_1,w_2,w_3\}$ as $u_i'$ does. On the other hand, if there are no such $u_1,u_1',\ldots,u_{t/2},u_{t/2}'\in V_2$, then 
    \begin{align*}
        |\mathcal E(\mathcal H)|\geq  |V_2|-\frac t2+1=r^{\frac 32+\frac 1{2\lfloor t/8 \rfloor}}-\frac t2+1=\Omega( |V(\mathcal H)|^{\frac 32+\frac 1{2\lfloor t/8\rfloor}} ).
    \end{align*} 
    By Lemma \ref{evencover} with $q=3$ and $k=t$, the hypergraph $\mathcal H$ contains an even-cover of size $t$, say formed by the edges $e_{v_1},\ldots,e_{v_t}$. Then $\{w_1,w_2,w_3\}$ and $\{v_1,\ldots,v_t\}$ will form an even-chromatic copy of $K_{3,t}$, because every vertex in $\mathcal H$ is covered by an even number of $e_{v_1},\ldots,e_{v_t}$ and hence in $K_n$
    every colour occurs an even number of times in the edge set $\{w_iv_j:i \in[3],j\in [t]\}$. 

    In conclusion, we can always find a $t$-vertex set $U\subseteq V_2$ such that $\{w_1,w_2,w_3\}$ and $U$ form an even-chromatic copy of $K_{3,t}$. However, 
    the $K_{s-3,t'}$ on $V_1\sqcup V_2$ is strongly-even-chromatic, and
    $U\subseteq V_2$, so the $K_{s-3,t}$ on $V_1\sqcup U$ is also strongly-even-chromatic.
    Thus the $K_{s,t}$ formed by $V_1\sqcup \{w_1,w_2,w_3\}$ and $U$ is even-chromatic.
\end{proof}

In the proof above, instead of taking $s'=s-3$, one can actually take $s'=s-d$ for any odd $d$ between $3$ and $s-2$, and make corresponding changes to run the same arguments. One can even take $d=s$, understood as fixing $s$ vertices $w_1,w_2,\ldots,w_s$ in the beginning and then taking $V_2=V(K_n)\setminus \{w_1,\ldots,w_s\}$ to perform the hypergraph even-cover argument directly.
Since $r_\text{odd}(n,K_{s,t})\leq O(n^{\frac{2s+2t-4}{st}})$, any choice of $d$ will give a lower bound that is log-asymptotically tight when $s$ is fixed and $t\to\infty$. However, the lower bound is optimised when $d=3$.

Although Theorem \ref{thmKST} determines $r_\text{odd}(n,K_{s,t})$ log-asymptotically for all fixed $s$ as $t\to \infty$, if both $s$ and $t$ are fixed, then there is still a polynomial gap between the current bounds. In particular, if $s$ is odd and $t\leq 6$, then $\lfloor t/8\rfloor=0$ 
and the lower bound in Theorem \ref{thmKST} is trivial, so the current best lower bound is still $r_\text{odd}(n,K_{s,t})\geq (1+o(1))(\frac nt)^{\lfloor s/2\rfloor}$ as in \cite{boyadzhiyska2024oddramseynumberscompletebipartite}. 

\section{Sparse odd-Ramsey numbers of Hamilton cycles}\label{sec3}

In this section we prove Theorem \ref{thmDIRAC}, showing that $r_\text{odd}(n,n/2+1; C_n) \ge 3$. In light of the construction of Boyadzhiyska et al.~\cite{boyadzhiyska2025oddramseynumbershamiltoncycles} of an $n$-vertex graph $G$ with $\delta(G) = n/2$ and $r_\text{odd}(G,C_n) = 2$, this is the smallest minimum degree that guarantees this nontrivial odd-Ramsey number for $C_n$.

We need to show that for $n\geq 3$ and any $n$-vertex graph $G$ with $\delta(G)\geq n/2+1$, any $2$-coloring of $E(G)$ yields an even-chromatic Hamilton cycle. 
Dirac \cite{dirac1952some} showed that any $n$-vertex graph $G$ with $\delta(G)\geq n/2$ contains a Hamilton cycle. Moreover, under the slightly stronger condition that $\delta(G)\geq (n+1)/2$, Ore \cite{ore1963hamilton} proved that $G$ is \textit{Hamilton-connected}, namely any pair of vertices is connected by a Hamilton path. We summarize these results in the following theorem.

\begin{theorem}\label{DiracOre}
    Let $n\geq 3$ and $G$ be an $n$-vertex graph.
    \begin{itemize}
        \item [(i)](Dirac) If $\delta(G)\geq n/2$, then $G$ contains a Hamilton cycle.
        \item [(ii)](Ore) If $d(u)+d(v)\geq n+1$ for all $u,v\in V(G)$ with $uv\notin E(G)$, then $G$ is Hamilton-connected. In particular, if $\delta(G)\geq (n+1)/2$, then any pair of vertices in $G$ is connected by a Hamilton path.
    \end{itemize}
\end{theorem}

We follow the same proof strategy as in~\cite{boyadzhiyska2025oddramseynumbershamiltoncycles}, as encapsulated in the following two propositions, but the implementation under the weaker minimum degree condition of $\delta(G) \ge n/2+1$ requires greater care.

\begin{proposition}\label{short-odd-cycles}
    Let $n\geq 6$ be an even integer and $G$ be an $n$-vertex graph with $\delta(G)\geq n/2+1$. Let $E(G)$ be $2$-coloured. If there is an odd-chromatic $4$- or $6$-cycle, then $G$ contains an even-chromatic Hamilton cycle.
\end{proposition}
\begin{proposition}\label{no-short-odd-cycles}
    Let $n\geq 6$ be an even integer and $G$ be an $n$-vertex graph with $\delta(G)\geq n/2+1$. Let $E(G)$ be $2$-coloured. If every $4$- and $6$-cycle is even-chromatic, then every Hamilton cycle in $G$ is even-chromatic.
\end{proposition}

These propositions combine to immediately give our result.

\begin{proof}[Proof of Theorem~\ref{thmDIRAC}]
Let $n \ge 6$ be an even integer, and let $G$ be an $n$-vertex graph with $\delta(G) \ge n/2 + 1$. Suppose we have a $2$-colouring of $E(G)$.

If there is any odd-chromatic $4$- or $6$-cycle, then by Proposition~\ref{short-odd-cycles}, there is an even-chromatic Hamilton cycle. Otherwise, if there is no odd-chromatic $4$- or $6$-cycle, then by Proposition~\ref{no-short-odd-cycles}, all Hamilton cycles --- which exist by Dirac's Theorem [Theorem \ref{DiracOre}(i)]--- are even-chromatic.

Hence, there is no $2$-colouring of $E(G)$ in which all Hamilton cycles are odd-chromatic, and thus $r_\text{odd}(G, C_n) \ge 3$. As $G$ was an arbitrary graph of minimum degree at least $n/2 + 1$, it follows that $r_\text{odd}(n,n/2+1;C_n) \ge 3$.
\end{proof}

We will now prove these propositions in turn.

\subsection{Proof of Proposition~\ref{short-odd-cycles}}

To prove Proposition~\ref{short-odd-cycles}, we need to show that, given a short odd-chromatic cycle, we can build an even-chromatic Hamilton cycle. Following~\cite{boyadzhiyska2025oddramseynumbershamiltoncycles}, we will do so building a \emph{switch}, which is a pair of Hamilton cycles whose symmetric difference is the short cycle in question. As this is odd-chromatic, it follows that the two Hamilton cycles must have different parities for the sizes of the colour classes, and hence one must be even-chromatic.

We begin with the case of the $4$-cycles.

\begin{lemma}\label{lem:4-cycle-switch}
    Let $n \ge 6$ be even and $G$ be an $n$-vertex graph with $\delta(G) \ge n/2+1$. Let $E(G)$ be $2$-coloured. If there is an odd-chromatic $4$-cycle, then $G$ contains an even-chromatic Hamilton cycle.
\end{lemma}

\begin{proof}
Let $abcda$ be an odd-chromatic $4$-cycle in our $2$-edge-coloured graph. As stated previously, our goal is to build a switch, and Figure~\ref{fig:4-cycle-switch} illustrates how we shall do so. We will find vertex-disjoint paths $P$ and $Q$ from $a$ to $c$ and from $b$ to $d$ respectively, whose union covers $V(G)$. We then have the  Hamilton cycles $ a b Q d c P a$ and $ a P c b Q d a$, whose symmetric difference is the odd-chromatic $abcda$. Thus, one of the two Hamilton cycles is even-chromatic, as desired.

\begin{figure}[H] 
        \centering
        \includegraphics[width=\linewidth]{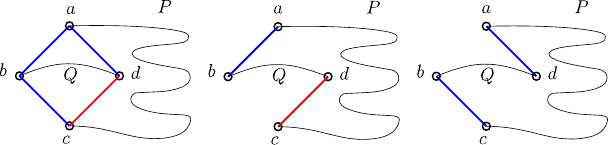}
        \caption{Building a switch out of an odd-chromatic $4$-cycle.}\label{fig:4-cycle-switch}
\end{figure}

In~\cite{boyadzhiyska2025oddramseynumbershamiltoncycles}, the switch was built by taking $Q$ to be a shortest $b,d$-path, and then taking $P$ to be an $a,c$-path that spans $G - Q$. To find $P$, the authors applied Ore's theorem on Hamilton-connected graphs, and suggested that to extend their result to graphs of lower minimum degree, one should strengthen Ore's result. We will instead take a different approach, allowing both $P$ and $Q$ to be long if needed.

First, we observe that there exists a pair of vertex-disjoint $a,c$- and $b,d$-paths. Indeed, since $\delta(G) \ge n/2 + 1$, it follows that any two vertices are either adjacent, or share at least $4$ common neighbours. Thus, for each pair $a,c$ and $b,d$ in turn, we either have an edge between them, or we can find a common neighbour not on the cycle that is distinct from any common neighbour chosen for the other pair. This yields vertex-disjoint $a,c$- and $b,d$-paths, as required.

Now, let $P$ and $Q$ be a pair of such paths maximising $v(P)+ v(Q)$. If $v(P) + v(Q) = n$, then these paths, together with the $4$-cycle $abcda$, provide the desired switch. Hence we may assume $v(P) + v(Q) < n$, and let $R = V(G) \setminus (V(P) \cup V(Q))$ be the set of remaining vertices not contained on either path. We write $p$, $q$, and $r$ for the number of vertices in $P$, $Q$, and $R$ respectively.

Given any path $\Gamma$, and a vertex $z \notin V(\Gamma)$, we say $\Gamma$ is \emph{$z$-extendable} if $z$ has neighbours $y_1, y_2$ appearing consecutively on $\Gamma$. We can then make $\Gamma$ longer by replacing the edge $y_1 y_2$ with the path $y_1 z y_2$. Note that our choice of the pair $P, Q$ implies that for every $z \in R$, neither $P$ nor $Q$ is $z$-extendable. This means that $z$ can at best be adjacent to every other vertex, leading to the following crucial observation.

\begin{observation} \label{path-extension}
    Let $\Gamma$ be a path with $\ell$ vertices, and let $z \notin V(\Gamma)$ be a vertex such that $\Gamma$ is not $z$-extendable. If $z$ is adjacent to $k$ of the endpoints of $\Gamma$, where $0 \le k \le 2$, then $z$ has at most $(\ell + k - 1)/2$ neighbours on $\Gamma$. 
\end{observation}

Now let $z \in R$ be an arbitrary vertex not on $P$ or $Q$. Since neither $P$ nor $Q$ is $z$-extendable, by Observation~\ref{path-extension}, we have $d_P(z) \le (p+1)/2$ and $d_Q(z) \le (q+1)/2$. As $d(z) \ge \delta(G) \ge n/2 + 1$, and $n = p + q + r$, it follows that we must have $d_R(z) \ge r/2$. Thus, $G[R]$ is an $r$-vertex subgraph of minimum degree at least $r/2$, and hence by Dirac's Theorem [Theorem \ref{DiracOre}(i)], there is a cycle $H$ spanning $R$ (where, if $r = 2$, we consider an edge to be a $2$-cycle). We now consider three cases.

\paragraph{Case 1: $r \ge n/2$.} Since $\delta(G) \ge n/2 + 1$, it follows that every vertex in $G$ has a neighbour in $R$. Let $u$ and $v$ be neighbouring points on the path $P$, and let $z_u, z_v \in R$ be neighbours of $u$ and $v$ respectively. We can then replace $P$ with a longer path $P'$, obtained by replacing the edge $uv$ with the edges $uz_u$ and $z_v v$, together with a segment of the cycle $H$ between $z_u$ and $z_v$. As the new vertices are all from $R$, this longer path $P'$ remains disjoint from $Q$, which contradicts our choice of $P$ and $Q$.

\paragraph{Case 2: $p,q \le r$.} We may assume that $r \le n/2 - 1$, as otherwise we are done by Case 1. This implies that every vertex $z \in R$ can have at most $n/2 - 2$ neighbours in $R$, and thus must have at least three neighbours in $V(P) \cup V(Q)$. In particular, there is a path that receives two edges from $z$.

Since $n = p + q + r$, and we assume $n \ge 6$, we must have $r \ge 3$. Let $w, x, y \in R$ be three consecutive vertices on the $R$-spanning cycle $H$. By the pigeonhole principle, two of these vertices, say $z_1$ and $z_2$, send two edges to the same path, say $P$. Note that, by taking the longer segment $\Gamma$ of $H$ between $z_1$ and $z_2$, we have a path in $R$ between these vertices that contains at least $r-1$ vertices.

Let $u_1$ be a vertex on $P$ adjacent to $z_1$, and let $u_2 \neq u_1$ be a distinct vertex on $P$ that is adjacent to $z_2$. We obtain a new path $P'$ from $P$ by replacing the segment of $P$ between $u_1$ and $u_2$ with the path $u_1 z_1 \Gamma z_2 u_2$. $P'$ contains at least $r+1$ vertices, and is still disjoint from $Q$, contradicting our choice of $P$ and $Q$.

\paragraph{Case 3: $\max (p,q) \ge r$.} Relabelling the paths if necessary, we may assume $p \ge r$. Repeating our degree calculation from before, for a vertex $z \in R$, we have $n/2 + 1 \le d(z) \le d_P(z) + d_Q(z) + d_R(z)$, which we rearrange to give $d_P(z) \ge n/2 - d_Q(z) - d_R(z) + 1$. We have $d_R(z) \le r-1$ trivially, and, since $Q$ is not $z$-extendable, $d_Q(z) \le (q+1)/2$. Since $n = p + q + r$, this yields $d_P(z) \ge (p - r + 3)/2$. As we are assuming $p \ge r$, it follows that every $z \in R$ has at least two neighbours on the path $P$.

Now let $x, y \in R$ be neighbours on the cycle $H$. Let $u$ be a neighbour of $x$ on the path $P$, and let $v \neq u$ be a distinct vertex on $P$ that is adjacent to $y$ and as close to $u$ as possible. We then can partition $P$ into three subpaths $P_1 \cup P_2 \cup P_3$, where $P_2$ is the segment of $P$ strictly between $u$ and $v$, and $P_1$ and $P_3$ are the segments between $u, v$ and the endpoints of $P$.

If $v(P_2) \le r-1$, we can replace $P_2$ with the path $H - xy$, which contains $r$ vertices. This yields a longer path $P'$ that is still disjoint from $Q$, again contradicting the choice of $P$ and $Q$. Hence, $v(P_2) \ge r$ and, by our choice of $v$, $y$ cannot be adjacent to any vertex on $P_2$. Thus, counting the degree of $y$, we have
\begin{align*}
    d(y) &= d_P(y) + d_Q(y) + d_R(y) = d_{P_1}(y) + d_{P_3}(y) + d_Q(y) + d_R(y) \\
    &\le (v(P_1) + 1)/2 + (v(P_3) + 1)/2 + (q + 1)/2 + (r-1) \\
    &= (p - v(P_2) + q + 2r + 1)/2 \\
    &\le (p + q + r + 1)/2 = (n+1)/2,
\end{align*}
where in the first inequality we ues the fact that none of $P_1$, $P_3$, and $Q$ can be $y$-extendable, and appeal to Observation~\ref{path-extension}. This contradicts $d(y) \ge \delta(G) \ge n/2 + 1$, completing our proof.
\end{proof}

Thus, if we have an odd-chromatic $4$-cycle, we are done. We next show that if there is no such $4$-cycle, then if we have an odd-chromatic $6$-cycle, we are also done. This will complete the proof of Proposition~\ref{short-odd-cycles}.

\begin{lemma} \label{lem:6-cycle-switch}
    Let $n \ge 6$ be even, and $G$ be an $n$-vertex graph with $\delta(G) \ge n/2 + 1$. Let $E(G)$ be $2$-coloured without any odd-chromatic $4$-cycle. If there is an odd-chromatic $6$-cycle, then $G$ contains an even-chromatic Hamilton cycle.
\end{lemma}

\begin{proof}
    Let $abcdefa$ be an odd-chromatic $6$-cycle in $G$. Our goal will again be to build a switch from this $6$-cycle; in this instance, the switch will consist of three disjoint paths covering all the vertices --- an $a,d$-path $P_0$, a $b,f$-path $P_1$, and a $c,e$-path $P_2$. Given such paths, we have the Hamilton cycles $abP_1feP_2cdP_0a$ and $afP_1bcP_2edP_0a$, whose symmetric difference is the odd-chromatic $6$-cycle. Thus, one of the two Hamilton cycles will have to be even-chromatic.

    \begin{figure}[H]
        \centering
        \includegraphics[width=0.85\linewidth]{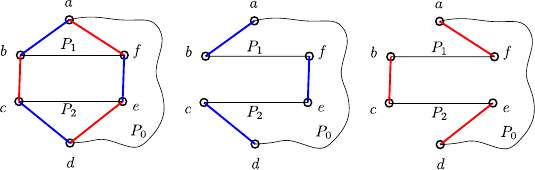}
        \caption{Building a switch out of an odd-chromatic $6$-cycle.} \label{fig:6-cycle-switch}
    \end{figure}
    
    We start by noting that the non-existence of odd-chromatic $4$-cycles imposes significant restrictions on what additional edges can be present in and around this $6$-cycle.

    \begin{observation} \label{obs:no-extra-edges}
        If $G$ has no odd-chromatic $4$-cycle, and $abcdefa$ is an odd-chromatic $6$-cycle, then we do not have:
        \begin{itemize}
            \item[(i)] the edge $ad$,
            \item[(ii)] the edges $ac$, $ae$, and $ce$, and
            \item[(iii)] a vertex $x$ with the edges $bx$, $dx$, and $fx$.
        \end{itemize}
    \end{observation}

    \begin{proof}[Proof of Observation~\ref{obs:no-extra-edges}]
        In all of these cases, we will express the $6$-cycle $abcdefa$ as the (modulo $2$) sum of various $4$-cycles. As every $4$-cycle is even-chromatic, it follows that the $6$-cycle must be too, giving a contradiction. For (i), we have the $4$-cycles $abcda$ and $adefa$. For (ii), we have the $4$-cycles $acefa$, $ceabc$, and $eacde$. Finally, for (iii), we take the cycles $xbcdx$, $xdefx$, and $xfabx$.
    \end{proof}

    Next, we show that there are vertex-disjoint paths $P_0$, $P_1$, and $P_2$ between the desired pair of vertices, perhaps after rotating the cycle.

    First, we claim that there is a pair of antipodal vertices on the $6$-cycle that have a common neighbour that is not on the cycle. Indeed, by Observation~\ref{obs:no-extra-edges}(i), $a$ and $d$ cannot be adjacent. Hence, each can have at most $4$ neighbours on the $6$-cycle, and thus have at least $n/2 - 3$ neighbours among the $n-6$ vertices not on the $6$-cycle. Thus, if $a$ and $d$ do not have a common neighbour outside the cycle, it follows that $a$ must be adjacent to both $c$ and $e$. Repeating this argument with $c$ in place of $a$, it follows that if there is no antipodal pair with a common external neighbour, then we have the edges $ac$, $ae$, and $ce$, violating Observation~\ref{obs:no-extra-edges}(ii). Hence we must have an antipodal pair of vertices on the cycle with a common external neighbour; we cyclically relabel the vertices on the cycle so that this pair is $a$ and $d$, and let $x_0$ be a common neighbour not on the cycle. We take $a x_0 d$ to be the path $P_0$.

    Now consider the pair $b, f$. If they are adjacent, we can take the edge $bf$ to be the path $P_1$. Otherwise, since $b$ cannot be adjacent to $d$ and $f$ cannot be adjacent to $c$ (by Observation~\ref{obs:no-extra-edges}(i)), each vertex has at most three neighbours on the $6$-cycle, and thus at least $n/2 - 2$ neighbours among the $n-6$ vertices not on the cycle. Thus, if $b$ and $f$ are not adjacent, they must have at least two common external neighbours. Furthermore, $x_0$ cannot be one of these neighbours, since in that case $x_0$ would violate Observation~\ref{obs:no-extra-edges}. We choose one of these common neighbours $x_1$ and take $P_1$ to be $b x_1 f$.

    By symmetry, either $c$ and $e$ are adjacent, in which case we take $P_2$ to be the edge $ce$, or they have two common neighbours not on the cycle, distinct from $x_0$. Hence we can find a common neighbour $x_2$ that is distinct from $x_1$, and take $P_2$ to be $c x_2 e$.

    This shows we can find vertex-disjoint paths $P_0$, $P_1$, and $P_2$ between $a$ and $d$, between $b$ and $f$, and between $c$ and $e$ respectively. We now consider such paths that maximise $v(P_0) + v(P_1) + v(P_2)$. If this sum is $n$, then these paths cover all the vertices of the graph, and we have our desired switch. Thus, we may assume $v(P_0) + v(P_1) + v(P_2) < n$, and let $R = V(G) \setminus (V(P_0) \cup V(P_1) \cup V(P_2))$ be the remaining vertices. For $0 \le i \le 2$, let $p_i = v(P_i)$, and let $r = |R|$.

    Now let $z \in R$ be arbitrary. By our choice of paths, for all $0 \le i \le 2$, the path $P_i$ cannot be $z$-extendable. Thus, by Observation~\ref{path-extension}, $d_{P_i}(z) \le (p_i + k_i - 1)/2$, where $k_i$ is the number of endpoints of $P_i$ that are adjacent to $z$. Thus, we have
    \begin{align*}
        n/2 + 1 \le d(z) &= d_{P_0}(z) + d_{P_1}(z) + d_{P_2}(z) + d_R(z) \\
        &\le \sum_{i=0}^2 (p_i + k_i - 1)/2 + d_R(z) \\
        &= (n-r)/2 + \sum_{i=0}^2 (k_i-1) /2 + d_R(z).
    \end{align*}
    It follows from Observation~\ref{obs:no-extra-edges}(iii), that $z$ can be adjacent to at most four vertices on the $6$-cycle, and hence $\sum_{i=0}^2 k_i \le 4$. This means we can rearrange the above to give $d_R(z) \ge (r+1)/2$, which by Ore's Theorem [Theorem \ref{DiracOre}(ii)] implies that $G[R]$ is Hamilton-connected.

    We once again complete the argument by considering three cases.

    \paragraph{Case 1: $r \ge n/2 - 1$.} Since $\delta(G) \ge n/2 + 1$, it follows that every vertex must have a neighbour in $R$. Let $u$ and $v$ be adjacent on $P_0$, and let $z_u, z_v \in R$ be their neighbours in $R$. Letting $\Gamma$ be a $z_u,z_v$-path in $R$, we can extend $P_0$ by replacing the edge $uv$ with the path $u z_u \Gamma z_v v$. This longer path remains disjoint from $P_1$ and $P_2$, contradicting our choice of the paths.

    Hence, we may assume that $r \le n/2 - 2$, which implies all vertices in $R$ have at least four neighbours in $P_0 \cup P_1 \cup P_2$. In particular, for each $z \in R$, there is some $0 \le i \le 2$ such that $z$ sends at least two edges to $P_i$.

    \paragraph{Case 2: There are $x, y \in R$ and some $0 \le i \le 2$ such that $d_{P_i}(x) \ge 2$ and $d_{P_i}(y) \ge 1$.} Relabelling the paths if needed, we assume $i = 0$. Let $u$ be a neighbour of $y$ on $P_0$, and let $v \neq u$ be the closest neighbour of $x$ on $P_0$ that is distinct from $u$. We partition $P_0$ into three subpaths $Q_1 \cup Q_2 \cup Q_3$, where $Q_2$ is the segment of $P_0$ strictly between $u$ and $v$, and $Q_1$ and $Q_3$ are the subpaths between $u, v$ and the endpoints of $P_0$.

    If $v(Q_2) < r$, then we can replace $u Q_2 v$ in $P_0$ with the path $u y \Gamma x v$, where $\Gamma$ is an $R$-spanning path between $x$ and $y$. This longer path is still disjoint from $P_1$ and $P_2$, contradicting our choice of paths.

    Thus, we must have $v(Q_2) \ge r$, and by our choice of $v$, none of these vertices are adjacent to $x$. Hence, $x$ can only send edges to $Q_1$, $Q_3$, $P_1$, $P_2$, and $R$, and none of the paths can be $x$-extendable. We can use Observation~\ref{path-extension} to bound the degree of $x$ into the paths. Since Observation~\ref{obs:no-extra-edges} implies $x$ can be adjacent to at most four vertices on the $6$-cycle, it follows that $x$ can be adjacent to a total of at most six of the endpoints of $Q_1$, $Q_3$, $P_1$, and $P_2$. Taken all together, this gives
    \begin{align*}
        d(x) &= d_{Q_1}(x) + d_{Q_3}(x) + d_{P_1}(x) + d_{P_2}(x) + d_R(x) \\
        &\le (v(Q_1) + v(Q_3) + v(P_1) + v(P_2) + 6 - 4)/2 + (r-1) \\
        &= (n-r-v(Q_2))/2 + r \le n/2,
    \end{align*}
    contradicting our minimum degree condition.

    \paragraph{Case 3: all other cases.} Since we have $d_R(z) \ge (r+1)/2$ for all $z \in R$, it follows that we must have $r \ge 3$. As we have previously established, since we are not in Case 1, each $z \in R$ sends at least two edges to some path $P_i$,
    and since we are not in Case 2, if some $z\in R$ sends at least two edges to $P_i$, then no other vertex in $R$ can send edges to $P_i$. Thus, we must have $R=\{z_0,z_1,z_2\}$, where for each $0\leq i\leq 2$ the neighborhood of $z_i$ is contained in $R\cup V(P_i)$.

    By the pigeonhole principle, there is some $i$ with $p_i \le (n-r)/3 = n/3 - 1$. Bearing in mind that $P_i$ is not $z_i$-extendable, we have
    \[ d(z_i) = d_{P_i}(z_i) + d_R(z_i) \le (p_i+1)/2 + (r-1) \le n/6 + 2, \]
    which, since $n \ge 6$, contradicts our minimum degree condition. This completes the proof.
    \end{proof}

\subsection{Proof of Proposition~\ref{no-short-odd-cycles}}

To prove Proposition~\ref{no-short-odd-cycles}, and with it Theorem~\ref{thmDIRAC}, we need to show that if our colouring of $G$ has no odd-chromatic $4$- or $6$-cycles, then every Hamilton cycle is even-chromatic.

Let $x$ and $y$ be two arbitrary vertices in $G$. We say that $x$ and $y$ \emph{agree} if, for any common neighbour $a$, the edges $ax$ and $ay$ have the same colour, and we say they \emph{disagree} if $ax$ and $ay$ have opposite colours. Note that this notion is well-defined, because $\delta(G) \ge n/2+1$ implies every pair of vertices have at least two common neighbours, and if there were common neighbours $a$ and $b$ such that $ax$ and $ay$ had the same colour but $bx$ and $by$ had opposite colours, then $xaybx$ would form an odd-chromatic $4$-cycle.

We next claim that \emph{agreeing} defines an equivalence relation on the vertices of $G$. For this, we shall use the non-existence of odd-chromatic $6$-cycles, and the following lemma will come in handy.

\begin{lemma} \label{lem:building-6-cycles}
Let $G$ be a graph with $\delta(G) \ge n/2 + 1$, and let $x,y,z$ be three vertices that are not pairwise adjacent. Then there are distinct vertices $a,b,c$ such that $xaybzcx$ forms a $6$-cycle in $G$.
\end{lemma}

\begin{proof}[Proof of Lemma~\ref{lem:building-6-cycles}]
We may assume that $x$ and $z$ are not adjacent. Then, if $V' = V(G) \setminus \{x,y,z\}$, we have $|V'| = n-3$, $d_{V'}(x) \ge n/2$, $d_{V'}(z) \ge n/2$, and $d_{V'}(y) \ge n/2 - 1$. It follows that $x$ and $y$ have at least two common neighbours in $V'$, as do $y$ and $z$, while $x$ and $z$ have at least three common neighbours in $V'$. Thus, we can pick $a \in V'$ to be a common neighbour of $x$ and $y$, choose $b \in V'$ to be a common neighbour of $y$ and $z$ distinct from $a$, and finally choose $c \in V'$ to be a common neighbour of $x$ and $z$ that is neither $a$ nor $b$. This gives the desired $6$-cycle.
\end{proof}

Now consider the relation of \emph{agreeing}. It is clearly reflexive and symmetric, so we need only show that it is transitive. To this end, suppose $x$ and $y$ agree, and $y$ and $z$ agree, but $x$ and $z$ disagree. Observe that we cannot have a triangle on $x$, $y$, and $z$, since $x$ and $y$ agreeing implies $xz$ and $yz$ must have the same colour, and similarly $y$ and $z$ agreeing implies $xy$ and $yz$ must have the same colour, and so the triangle would have to be monochromatic. But then $xy$ and $yz$ have the same colour, so $x$ and $z$ would also agree, contrary to our assumption.

Thus, $x$, $y$, and $z$ are not pairwise adjacent, and so by Lemma~\ref{lem:building-6-cycles}, we find a $6$-cycle $xaybzcx$. Since $x$ and $y$ agree, $xa$ and $ay$ have the same colour. Similarly, $yb$ and $bz$ must have the same colour. But since $x$ and $z$ disagree, $zc$ and $cx$ must have opposite colours. This implies the $6$-cycle is odd-chromatic, which is a contradiction.

Hence, \emph{agreeing} defines an equivalence relation. Now suppose we had at least three equivalence classes, and let $x$, $y$, and $z$ come from distinct classes --- that is to say, they pairwise disagree. We again cannot have a triangle on $x$, $y$, and $z$, as if we did, two edges, say $xy$ and $yz$, would have to be of the same colour, which contradicts the assumption that $x$ and $z$ disagree.

Therefore these three vertices span at most two edges, and applying Lemma~\ref{lem:building-6-cycles}, we again find a $6$-cycle $xaybzcx$. Since $x$, $y$, and $z$ pairwise disagree, we have that $xa$ and $ay$ have opposite colours, as do $yb$ and $bz$, and $zc$ and $cx$. Thus, the $6$-cycle is odd-chromatic, which is again a contradiction.

Thus, there are at most two equivalence classes. Now, given any Hamilton cycle $H = v_1 v_2 v_3 \hdots v_{n-1} v_n v_1$, consider the sequence $v_1, v_3, \hdots, v_{n-1}, v_1$.. As we cycle through these odd-indexed vertices, we must cross between the two equivalence classes an even number of times. Whenever two consecutive vertices $v_{2i-1}, v_{2i+1}$ come from the same equivalence class, they agree, and so the two edges of $H$ between them, namely $v_{2i-1} v_{2i}$ and $v_{2i} v_{2i+1}$, must have the same colour. On the other hand, when the vertices come from different classes, they disagree, and so the edges $v_{2i-1} v_{2i}$ and $v_{2i} v_{2i+1}$ have opposite colours. As this happens an even number of times, though, it follows that the cycle $H$ is even-chromatic, as required. This completes the proof of Proposition~\ref{no-short-odd-cycles}.

\section{Unique-Ramsey numbers of Hamilton cycles}\label{sec4}

In this section we prove Theorem \ref{thmUNIQ}. Namely, if $r\leq n/4$, then in every $r$-colouring of $K_n$, we can find a Hamilton cycle in which no colour occurs exactly once. Roughly speaking, the proof idea is as follows.

In the beginning, we label all the colours as \textit{unused}. Then we find a maximum collection of pairwise disjoint monochromatic cherries or 2-matchings, the members of this collection having pairwise distinct colours.
	We preserve these structures and relabel their colours as \textit{free}, in the sense that our final Hamilton cycle will go through each of these cherries or 2-matchings, so their colours are guaranteed to occur at least twice and we do not need to worry about them in the subsequent steps. Moreover, the maximality of this collection guarantees that within the remaining vertices, each unused colour can occur on at most one edge, so it is also easy to find an almost spanning structure in this remaining part with only free colours. Now we only need to merge the preserved structures and the remaining part into a Hamilton cycle, making sure that we do not introduce any unique unused colour throughout the process.
    
    However, edges between the preserved structures and the remaining part can be problematic, since we have no control over the occurrence of unused colours here. To resolve this problem, in the first step we shall construct the maximum collection more carefully by considering claws, namely copies of $K_{1,3}$, each of which will then be split into a cherry and a singleton and added to the collection of preserved structures.
    Moreover, 
	in the merging step we shall also first merge the preserved structures so that fewer mergings with the remaining part are to be considered.

\begin{proof}[\proofname\mbox{}   of Theorem \ref{thmUNIQ}]
Let $r\leq n/4$ and let an $r$-colouring of $K_n$ be given. The following steps find a Hamilton cycle in which no colour occurs exactly once. Throughout the process, we shall maintain a set $U$ of colours, initialised to $U=[r]$. colours in $U$ are called \textit{unused}, and colours not in $U$ are called \textit{free}.\\\\
\textbf{Step 1: Pairwise disjoint monochromatic claws with pairwise distinct colours}\\\\
	First, let $\mathcal C$ be a maximum collection of pairwise disjoint monochromatic claws in $K_n$ that have pairwise distinct colours. That is,
	\begin{enumerate}
		\item  Each $H\in \mathcal C$ is a monochromatic copy of $K_{1,3}$.
		\item  Every pair of distinct $H_1,H_2\in\mathcal C$ are disjoint and have different colours.
	\end{enumerate}
	Let $|\mathcal C|=s$ and $\mathcal C=\{H_1,\ldots,H_s\}$, where $V(H_i)=\{x_0^{(i)},x_1^{(i)},x_2^{(i)},x_3^{(i)}\}$ with $x_1^{(i)},x_2^{(i)},x_3^{(i)}$ being the leaves.
    From $U$ we remove the colours that occur in $\mathcal C$.
    Let $R=V(K_n)\setminus (V(H_1)\sqcup\ldots\sqcup V(H_i))$.

    For any vertex $v\in V(K_n)$, we say that $v$ is \textit{dangerous} if it is the center of a monochromatic claw whose colour is in $U$ and leaves are in $R$. By the maximality of $\mathcal C$, there is no dangerous vertex in $R$.
    
    Next, we consider a collection $\mathcal P$ of subgraphs of $K_n$, initialised to $\mathcal P=\emptyset$. This $\mathcal P$ is the collection of structures to be preserved.
    Letting $i$ range from $1$ to $s$, we examine each claw $H_i$ and do the following.
	\begin{enumerate}
		\item  If none of $x_1^{(i)},x_2^{(i)},x_3^{(i)}$ are dangerous, then we do nothing and examine the next claw.
		\item  Otherwise we may assume without loss of generality that $x_1^{(i)}$ is dangerous, say $\{x_1^{(i)},y_1,y_2,y_3\}$ forms a monochromatic claw in some unused colour $c\in U$ with leaves $y_1,y_2,y_3\in R$. Suppose there is 
        some other unused colour $c'\in U$, $c'\ne c$ and $z_1,z_2,z_3\in R\setminus \{y_1,y_2,y_3\}$ such that $\{x_2^{(i)},z_1,z_2,z_3\}$ or $\{x_3^{(i)},z_1,z_2,z_3\}$ forms a monochromatic claw $K$ in colour $c'$ with leaves $z_1,z_2,z_3$, then by replacing $H_i$ with $\{x_1^{(i)},y_1,y_2,y_3\}$ and $K$ from $\mathcal C$, we obtain a larger collection of monochromatic claws in pairwise distinct colours, contradicting the maximality of $\mathcal C$. Thus such $c'$ and $z_1,z_2,z_3$ cannot exist.
        
        From the collection $\mathcal C$, we replace $H_i$ by the claw on $\{x_1^{(i)},y_1,y_2,y_3\}$ centered at $x_1^{(i)}$, and we add the monochromatic cherry $x_2^{(i)}x_0^{(i)}x_3^{(i)}$ to the collection $\mathcal P$. We remove $y_1,y_2,y_3$ from $R$ since they are now contained in structures from $\mathcal C$, and remove $c$ from $U$ so it is now a free colour. By our previous discussion, now none of $y_1,y_2,y_3,x_2^{(i)},x_3^{(i)}$ are dangerous, and we examine the next claw in $\mathcal C$. 
	\end{enumerate}
    Now $\mathcal C$ is still a collection of $s$ monochromatic claws $H_1',\ldots,H_s'$,  pairwise disjoint and in distinct colours, but with the extra property that none of the leaves are dangerous. Let each $V(H_i')=\{y_0^{(i)},y_1^{(i)},y_2^{(i)},y_3^{(i)}\}$ with $y_1^{(i)},y_2^{(i)},y_3^{(i)}$ being the leaves.
    Also, now $\mathcal P$ is a collection of monochromatic cherries, none of whose endpoints are dangerous.
    \begin{figure}[H]
        \centering
        \includegraphics[width=0.9\linewidth]{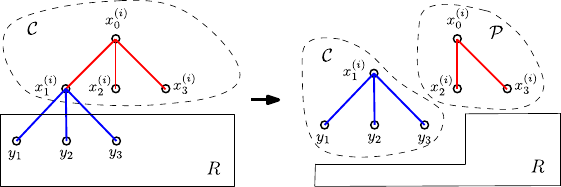}
        \caption{Updating the collections $\mathcal C$ and $\mathcal P$. }
        \label{updateC}
    \end{figure}
    \textbf{Step 2: Extending the collection $\mathcal P$ by cherries, 2-matchings and singletons}\\\\
	We order the colours in $U$ and examine them one by one.
	For each $c\in U$, if some $z_1,z_2,z_3\in R$ form \linebreak a monochromatic cherry with colour $c$, then we add that cherry to $\mathcal P$, remove the vertices $z_1,z_2,z_3$ from $R$ and remove $c$ from $U$. Otherwise, if some $z_1,z_2,z_3,z_4\in R$ form a monochromatic 2-matching with colour $c$, say on edges $z_1z_2$ and $z_3z_4$, then we add the edges $z_1z_2$ and $z_3z_4$ to $\mathcal P$, remove the vertices $z_1,z_2,z_3,z_4$ from $R$ and remove $c$ from $U$.
    
	After this process, every still unused colour $c\in U$ can occur on at most one edge in $R$. Moreover, letting $t$ be the number of colours in $\mathcal P$ at this point, then $|U|=r-s-t$ and for some $3\leq \alpha\leq 4$ we have $|R|=n-4s-\alpha t$. Note that our initial assumption $r\leq n/4$ guarantees that
	\begin{align*}
		|R|=n-4s-\alpha t\geq 4r-4s-4t=4|U|.
	\end{align*}
	Next we break the claws in $\mathcal C$ and add them to $\mathcal P$. Namely, for each $1\leq i\leq s$, we add the singleton $y_1^{(i)}$ and the monochromatic cherry $y_2^{(i)}y_0^{(i)}y_3^{(i)}$ to $\mathcal P$. Now we can write $\mathcal P=\{S_1,\ldots,S_s,G_{s+1},\ldots,G_{m}\}$ for some $m$, where each $S_i$ is a singleton and each $G_i$ is either a cherry or an edge. By our construction, neither an $S_i$ nor an endpoint of $G_i$ is dangerous. \\\\
	In $\mathcal P$, each $G_i$ is a path and we let $u_i,v_i$ be its endpoints. As for the singletons $S_j=\{u_j\}$, we introduce $s$ additional vertices $v_1,\ldots,v_s$ into the graph by defining each $N(v_j)=R\cup \{u_j\}$, where for all $z\in R$ we let the edge $v_jz$ receive
	the same colour as $u_jz$, and let $G_j=u_jv_j$ receive any free colour. Under this modification, it suffices to find a Hamilton cycle containing the paths $G_1,\ldots,G_m$ such that no colour in this Hamilton cycle occurs exactly once. To do this, it suffices to extend $G_{1},\ldots,G_m$ to a Hamilton cycle by a set of edges in which every unused colour $c\in U$ does not occur exactly once. 
	Note that if $U=\emptyset$, then this can be done trivially, so from now on we assume $|U|\geq 1$.
    \begin{figure}[H]
        \centering
        \includegraphics[width=\linewidth]{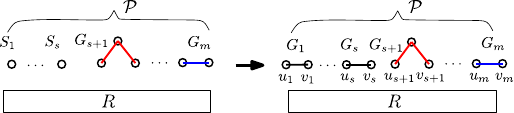}
        \caption{The modification of $\mathcal P$. }
    \end{figure}
\textbf{Step 3: Merging paths in $\mathcal P$ using edges between endpoints}\\\\
	We first merge some of the paths $G_1,\ldots, G_m$ using only edges between their endpoints. In the beginning, let the set of endpoints of these paths be $\mathcal E=\{u_1,v_1,\ldots,u_m,v_m\}$. Now we do the following.
	\begin{enumerate}
		\item  If there exist $w_1,w_2\in \mathcal E$ and distinct $F_{i_1},F_{i_2}\in\mathcal P$
        such that $w_1$ is an endpoint of $F_{i_1}$, $w_2$ is an endpoint of $F_{i_2}$,  and
		$w_1w_2$ has colour not in $U$, then from $\mathcal P$ we replace $F_{i_1}$ and $F_{i_2}$ by the merged path $F_{i_1}w_1w_2 F_{i_2}$, and remove $w_1,w_2$ from $\mathcal E$. We repeat this step until no such $w_1,w_2$ exist.
		\item If there exist distinct $w_1,w_2,w_3,w_4\in \mathcal E$ 
        and distinct $F_{i_1},F_{i_2},F_{i_3},F_{i_4}\in \mathcal P$ such that each of $w_1,w_2,w_3,w_4$ 
        is an endpoint of $F_{i_1},F_{i_2},F_{i_3},F_{i_4}$, respectively,  and $w_1w_2$ has the same colour as $w_3w_4$, then from $\mathcal P$ we replace $F_{i_1},F_{i_2},F_{i_3},F_{i_4}$ by the merged paths $F_{i_1}w_1w_2F_{i_2}$ and $F_{i_3}w_3w_4F_{i_4}$, remove $w_1,w_2,w_3,w_4$ from $\mathcal E$, and remove the colour on $w_1w_2$ from $U$. Then we restart from (i). 
        
        Likewise, if there exist distinct $w_1,w_2,w_3,w_4\in \mathcal E$ 
        and distinct $F_{i_1},F_{i_2},F_{i_3}\in \mathcal P$
        such that each of $w_1,w_2,w_3,w_4$ 
        is an endpoint of $F_{i_1},F_{i_2},F_{i_2},F_{i_3}$, respectively,  and $w_1w_2$ has the same colour as $w_3w_4$, then from $\mathcal P$ we replace $F_{i_1},F_{i_2},F_{i_3}$ by the merged path $F_{i_1}w_1w_2F_{i_2}w_3w_4 F_{i_3}$, remove $w_1,w_2,w_3,w_4$ from $\mathcal E$, and remove the colour on $w_1w_2$ from $U$. Then we restart from (i).
        
        We repeat until no  $w_1,w_2,w_3,w_4$ satisfy the above conditions.
	\end{enumerate}
	After this process, let $\mathcal P=\{G_1',\ldots,G_{m'}'\}$, then no colour in $G_1'\sqcup \ldots\sqcup G_{m'}'$ occurs exactly once. Each $G_i'$ is a path, say it has endpoints $u_i',v_i'$. Consider the subgraph $H$ on vertices $u_1',v_1',\ldots,u_{m'}',v_{m'}'$ containing all the edges aside from $u_1'v_1',\ldots,u_{m'}v_{m'}$ and inheriting the original colouring. Since the above process has halted, from (i) we know that $H$ sees only colours in $U$, and from (ii) we know that each colour class is either a star or contained within $\{u_i',v_i',u_j',v_j'\}$ for some $i,j$. Define
	\begin{align*}
		I=\{i\in [m']:\text{some }w_i\in\{u_i',v_i'\}\text{ is not the center of a monochromatic star in }H\}.
	\end{align*}
	For all distinct $i,j\in I$, let $w_iw_j$ receive the colour $c$, then the colour class $c$ cannot be a star since neither $w_i$ nor $w_j$ can be the center, so the colour class $c$ is contained within $\{u_i',v_i',u_j',v_j'\}$. Therefore, among $\{u_i'v_i':i\in I\}$ we see at least $\binom {|I|}{2}$ colours. On the other hand, if $i\notin I$, then there is a colour $c$ whose class forms a star centered at $u_i$, and likewise for $v_i$, so we see $2(m'-|I|)$ more different colours. Thus by counting the colours occuring in $H,$ we get
	\begin{align*}
		2m'-3\leq \binom {|I|}{2}+2(m'-|I|)\leq |U|,\quad 2m'\leq |U|+3.
	\end{align*} 
    For any $u_i'$ and $z\in R$, we say that $z$ is a \textit{free neighbor} of $u_i'$ if the colour on $u_i'z$ is free, and 
    we define the same term for $v_i'$. Recall that none of $u_1',v_1',\ldots,u_{m'}',v_{m'}'$ are dangerous, so each of them has at least $|R|-2|U|$ free neighbors in $R$.
    
	\textbf{Special case: $|U|=1$}\\\\
	The case $|U|=1$ is simple but needs to be separated. In this case we have $m'\leq 2$. Let  $U=\{c\}$.  Note that $|R|\geq 4|U|\geq 4$, and there is at most one edge in $R$ coloured $c$.

    If $m=1$, then $\mathcal P=\{G_1'\}$, and $G_1'$ has endpoints $u_1',v_1'$. Each of them has at least $|R|-2$ free neighbors in  $R$, so in particular we can find distinct $z_1,z_2\in R$ such that both $u_1'z_1$ and $v_1'z_2$ are not coloured $c$. There is a Hamilton $\{z_1,z_2\}$-path $P$ in the subgraph induced by $R$ avoiding the colour $c$, so a required Hamilton cycle is given by
    \begin{align*}
        v_1'G_1'u_1'z_1Pz_2v_1'.
    \end{align*}
    If $m'=2$, then the paths in $\mathcal P$ are $G_1',G_2'$, having endpoints $u_1',v_1'$ and $u_2',v_2'$, respectively. Note that every edge between $\{u_1',v_1'\}$ and $\{u_2',v_2'\}$ is coloured $c$, otherwise we should have merged $G_1',G_2'$ and get $m'=1$.
	\begin{enumerate}
		\item  If each of $u_1',v_1',u_2',v_2'$ has only free neighbors in $R$, then 
        we can find distinct $z_1,z_2,z_3,z_4\in R$ such that none of $u_1'z_1,v_1'z_2,u_2'z_3,v_2'z_4$ are coloured $c$. Now we can find two disjoint paths $P,Q$ from $\{z_1,z_2\}$ to $\{z_3,z_4\}$ in the subgraph induced by $R$ avoiding the colour $c$ such that $V(P)\sqcup V(Q)=R$, and hence by joining the paths
		\begin{align*}
			z_2v_1'G_1'u_1'z_1,\quad z_3u_2' G_2'v_2'z_4 
		\end{align*}
		with $P$ and $Q$, we get the required Hamilton cycle.
        \item Otherwise we may assume without loss of generality that $u_1'z$ has colour $c$ for some $z\in R$,
        then we can extend $G_1',G_2'$ to $zu_1'G_1'v_1'u_2'G_2'v_2'$. This uses the colour $c$ twice, so
		we can then close up any Hamilton cycle, as needed. 
	\end{enumerate}
	Thus from now on we assume $|U|\geq 2$.\\\\
    \textbf{Step 4: Merging paths in $\mathcal P$ using cherries centered in $R$}\\\\
	From the collection $\mathcal P=\{G_1',\ldots,G_{m'}'\}$, we further merge the paths using suitable cherries centered  in $R$. To do this, we define the set $C$ of used cherry centers in $R$ and set $C=\emptyset$ initially. Then we do the following, where $|C|\leq m'-2\leq (|U|-1)/2$ throught out the process. 
	\begin{enumerate}
		\item For all $w\in u_1',v_1',\ldots,u_{m'}',v_{m'}'$, since $w$ has at least $|R|-2|U|$ free neighbors in $R$, we have
        \begin{align*}
            \left|\Big\{\text{free neighbors of }w\text{ in }R\setminus C\Big\}\right|\geq |R|-2|U|-|C|\geq |R|-2|U|-\frac{|U|-1}2=|R|-\frac{5|U|-1}2.
        \end{align*}
        Whenever we have distinct $F_{i_1}', F_{i_2}', F_{i_3}'\in\mathcal P$, we pick $w_{i_1},w_{i_2},w_{i_3}$ from the endpoints of $F_{i_1}', F_{i_2}', F_{i_3}'$, respectively. Since $w_{i_1},w_{i_2},w_{i_3}\in \{u_1',v_1',\ldots,u_{m'}',v_{m'}'\}$, each of them has at least $|R|-(5|U|-1)/2$ free neighbors in $R\setminus C$. 
		Moreover, since $|R|\geq 4|U|$, we get
		\begin{align*}
			3\left( |R|-\frac{5|U|-1}{2} \right)>|R|,
		\end{align*}
		so at least two of $w_{i_1},w_{i_2},w_{i_3}$, say $w_{i_1}$ and $w_{i_2}$ without loss of generality, have a common free neighbor $v$ in $R\setminus C$. Then from $\mathcal P$ we replace $F_{i_1}'$ and $F_{i_2}'$ by the merged path $F_{i_1}'w_{i_1}vw_{i_2}F_{i_2}'$, and add $v$ to $C$.
		\item We stop when the original paths $G_1',\ldots,G_{m'}'$ in $\mathcal P$ are merged into two paths $F_1'$ and $F_2'$ by the cherries centered in $C$. Let $F_1'$ have endpoints $w_1,w_1'$ and $F_2'$ have endpoints $w_2,w_2'$. Note that at this point $|C|=m'-2\leq (|U|-1)/2$, so since $|R|\geq 4|U|$ and $|U|\geq 2$,
		\begin{align*}
			\left|\Big\{\text{free neighbors of }w_1\text{ in }R\setminus C\Big\}\right|\geq |R|-\frac{5|U|-1}{2}\geq 3.5,
		\end{align*}
        and likewise each of $w_1',w_2,w_2'$ has at least $3.5$ free neighbors in $R\setminus C$. Therefore,
        we can find pairwise distinct $z_1,z_1',z_2,z_2'\in R\setminus C$ such that $w_1z_1,w_1'z_1',w_2z_2,w_2'z_2'$ all have colours not in $U$.
	\end{enumerate}
    \begin{figure}[H]
        \centering
        \includegraphics[width=0.76\linewidth]{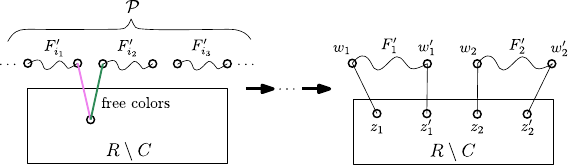}
        \caption{Merging paths in $\mathcal P$ using cherries centered in $R$. }
        \end{figure}
    \textbf{Step 5: Close up the Hamilton cycle}\\\\
	Now it remains to find a $\{z_1,z_2\}$-path $P$ and a $\{z_1',z_2'\}$-path $Q$ in the subgraph induced by $R\setminus C$, $P$ and $Q$ being disjoint, covering $R\setminus C$ and avoiding colours in $U$, because then the required Hamilton cycle is given by 
	\begin{align*}
		z_1w_1 F_1'w_1'z_1' Qz_2'w_2'F_2'w_2z_2Pz_1.
	\end{align*}
	To do this, consider two additional vertices $z_1^*,z_2^*$ and the graph $G$ defined by
	\begin{align*}
		&V(G)=\{z_1^*,z_2^*\}\cup (R\setminus C)\setminus\{z_1,z_1',z_2,z_2'\},\quad E(G)=E_1\cup E_2,\\
		&E_1=\Big\{e\in E\Big(K_n\Big[(R\setminus C)\setminus\{z_1,z_1',z_2,z_2'\}\Big]\Big): e\text{ has colour not in }U\Big\},\\
		&E_2=\big\{ z_1^*v: z_1v,z_1'v\text{ have colours not in }U \big\}\cup \big\{ z_2^*v: z_2v,z_2'v\text{ have colours not in }U \big\}.
	\end{align*}
	Then it suffices to find a Hamilton cycle in $G$, because by expanding $z_1^*,z_2^*$ back to $z_1,z_1',z_2,z_2'$ we can recover the $P$ and $Q$ back in $R\setminus C$. A Hamilton cycle in $G$ can be guaranteed by checking Dirac's condition.
    
    We have $v(G)=|R|-|C|-2\geq (7|U|-3)/2$. Note that each colour from $U$ is in at most one edge in  $R$,  so each $v\in (R\setminus C)\setminus \{z_1,z_1',z_2,z_2'\}$ has at most  $|U|$ nonneighbors in $G$, and indeed $d_G(v)\geq v(G)/2$. \linebreak On the other hand, let 
	\begin{align*}
		A_1=\big\{ v\in R:z_1v\text{ has colour in }U \big\},\quad A_2=\big\{ v\in R:z_1'v\text{ has colour in }U \big\},
	\end{align*}
	then by the same reason, $|A_1\cup A_2|\leq |U|$, and hence
	\begin{align*}
		d_G(z_1^*)=|V(G)\cap A_1^c\cap A_2^c|-1\geq v(G)-|A_1\cup A_2|-1\geq v(G)-|U|-1.
	\end{align*}
    Since $v(G)\geq (7|U|-3)/2$, we have $v(G)-|U|-1\geq v(G)/2$ when $|U|\geq 3$. Moreover, if $|U|=2$, then $v(G)\geq (7|U|-2)/2$ and hence we still get $v(G)-|U|-1\geq v(G)/2$. Therefore, $d_G(z_1^*)\geq v(G)/2$.
	Likewise $d_G(z_2^*)\geq v(G)/2$, so Dirac's condition is verified.\qedhere
\end{proof}

The condition $r\leq n/4$ is tight in the proof. To see this, note that in the first step, after picking the collection $\mathcal C=\{H_1,\ldots,H_s\}$ of pairwise disjoint monochromatic claws with pairwise distinct colours, we might get $s=n/4$.
When $r\leq n/4$, this means that $r=n/4$ and $R=\emptyset$, $U=\emptyset$, which is trivial since we can then break these claws into cherries and singletons and extend them into a Hamilton cycle arbitrarily.

However,
if $r>n/4$, then this means that all vertices are contained in some member of $\mathcal C$ but there are still $r-n/4$ unused colours. We cannot tell how these unused colours occur on $E(K_n)\setminus (E(H_1)\cup\ldots\cup E(H_s))$, so even with $R=\emptyset$ it can be hard to find a required Hamilton cycle.

As stated before the proof, claws are introduced in the beginning to ensure that unused colours behave well in the interface between the preserved structures and remaining vertices. This is realised by the process in Figure \ref{updateC}. That is, when a claw $H\in\mathcal C$ has a dangerous leaf, say that leaf is the center of a claw $K$ whose colour is unused and leaves are in $R$, then in $\mathcal C$ we replace $H$ by $K$, and reduce $H$ to a cherry and put it into $\mathcal P$. The colour on $H$ is still free since it still occurs twice in a preserved cherry, so the new claw $K\in \mathcal C$ will have no dangerous leaf. Repeatedly, $\mathcal C$ becomes a monochromatic claw collection with no dangerous leaf.\

In contrast, if we do not consider claws, so $\mathcal C$ is just a collection of monochromatic cherries or 2-matchings in the beginning, then the process in Figure \ref{updateC} cannot be carried out, because if we replace an $H\in \mathcal C$ by another structure intersecting a leaf of $H$, then the colour on $H$ might no longer be free. Thus we cannot rule out the possibility that a leaf in $\mathcal C$ sees many unused colour to the remaining vertices, and that obstructs the proof.

That being said, in \cite{boyadzhiyska2025oddramseynumbershamiltoncycles} the upper bound $r_{\text{odd}}(n,C_n)\leq \frac{3\sqrt 2}2\sqrt n$ on the odd-Ramsey number of Hamilton cycles was given. Therefore, our lower bound $r_\text{u}(n,C_n)> n/4$ already implies a polynomial gap between the odd- and unique-Ramsey numbers of Hamilton cycles.
Moreover, when we consider the unique-Ramsey number itself, the current best upper bound is $r_\text{u}(n,C_n)\leq n/2+1$ as in Proposition \ref{Upper}, so $r_\text{u}(n,C_n)$ is determined up to a multiplicative factor of $2$.
\section{Concluding remarks}\label{sec5}

In this paper, we studied the odd-Ramsey number of fixed complete bipartite graphs, the odd-Ramsey number of Hamilton cycles when the host graph is a super-Dirac graph, and the unique-Ramsey number of Hamilton cycles. Here we list some possible further directions and related problems.

\textbf{Closing the gap between the bounds on $r_{\text{odd}}(n,K_{s,t})$}

Together with previous results, our lower bound in Theorem \ref{thmKST} shows that for each fixed $s$, as $t\to\infty$ while keeping $st$ even, the odd-Ramsey number $r_\text{odd}(n,K_{s,t})$ is log-asymptotically tight. Howevir, for each fixed $s$ and $t$ such that $st$ is even, there is still a polynomial gap between the current best bounds, namely
\begin{align*}
    O\left(n^{\frac{2s+2t-4}{st}}\right)\geq r_\text{odd}(n,K_{s,t})\geq \begin{cases}
        (1+o(1))\left(\frac nt\right)^{2/s}, & \text{if } s \text{ is even, or }s\text{ is odd and }t\leq 6\\
        n^{\left(\frac s2+\frac{1}{2\lfloor t/8\rfloor}\right)^{-1}}, & \text{if } s \text{ is odd and }t\geq 8.
    \end{cases}
\end{align*}

\begin{problem}
    Let $s,t\geq 3$ be fixed integers and $st$ be even. What is the asymptotic behavior of $r_\mathrm{odd}(n,K_{s,t})$?
\end{problem}

$O\left(n^{\frac{2s+2t-4}{st}}\right)$ is actually an upper bound on $f(n,K_{s,t},n/2+1)$, and is obtained by the Lovász local lemma, and it appears difficult to use this approach to get a sharper bound for the odd-Ramsey setting. Therefore, to improve the upper bound, one might need an explicit construction or a more delicate probabilistic argument.

On the other hand, the lower bounds are obtained by a Kővári-Sós-Turán type argument and results on hypergraph even-covers. Currently there are still some wasteful steps in the proofs. For instance, when $s$ is even, the lower bound given by the Kővári-Sós-Turán type argument in \cite{boyadzhiyska2024oddramseynumberscompletebipartite} is actually preventing a strong-even-chromatic $K_{s,t}$, which is stronger than needed. On the other hand, when $s$ is odd,
in the proof of Theorem \ref{thmKST}, we fixed $w_1,w_2,w_3$ outside of the strong-even-chromatic $K_{s-3,t}$ in order to define our auxiliary hypergraph, and this seems to be too restrictive. Therefore, an improvement on the lower bounds might be possible.

\textbf{The sparse-odd-Ramsey number of Hamilton cycles}

We know that $r_\text{odd}(n,n/2;C_n)=2$, and prove in Theorem \ref{thmDIRAC} that $r_\text{odd}(n,n/2+1;C_n)\geq 3$. This shows that while the odd-Ramsey number can surprisingly be trivial for graphs of minimum degree $n/2$, it jumps to a non-trivial value at the next possible opportunity. That said, we fall short of determining $r_\text{odd}(n,n/2+1;C_n)$; Boyadzhiyska et al.~\cite{boyadzhiyska2025oddramseynumbershamiltoncycles} provide a construction to show $r_\text{odd}(n, n/2+1; C_n) \leq 4$, and we have not been able to improve upon this.

We tend to believe their upper bound is correct, so improvement will require adapting the machinery for the lower bound to handle more colours. Recall that to prove Theorem~\ref{thmDIRAC}, we had to find an even-chromatic Hamilton cycle in two-colourings of graphs, for which it sufficed to build a \emph{switch} that swapped the parities of the colour classes. When there are three or more colours involved, one may instead need to find multiple switches for the different colours, and manage to combine them into one spanning structure. As a test case, it would be interesting to prove a lower bound of $4$ for slightly less sparse Dirac graphs.

\begin{problem}
    Is there some constant $k \ge 1$ such that $r_\text{odd}(n,n/2+k;C_n) \geq 4$ for all even $n \ge 6$? Do we in fact have $r_\text{odd}(n,n/2+1;C_n) = 4$? 
\end{problem}

The ultimate goal is to understand the true behavior of $r_\text{odd}(n,n/2+k;C_n)$ as a function of $k$. In particular, the upper bound $r_\text{odd}(n,n/2+k;C_n)\leq \min\{2k+2,\frac{3k}{\sqrt {2n}}+\frac{3\sqrt{2n}}4+3\}$ in \cite{boyadzhiyska2025oddramseynumbershamiltoncycles} grows linearly in $k$ with slope $2$ for $k=O(\sqrt n)$ and with slope $3/\sqrt{2n}$ for larger $k$. Aside from improving the lower bounds, it would also be interesting to prove better upper bounds for these ranges of $k$.

\begin{problem}
    Let $n\geq 4$ be an even integer. How does $r_\mathrm{odd}(n,n/2+k;C_n)$ grow as $k$ ranges from $1$ to $n/2$?
\end{problem}

\textbf{The odd- and unique-Ramsey numbers of Hamilton cycles and Hamilton paths}

With Theorem \ref{thmUNIQ} now we know $n/4< r_\text{u}(n,C_n)\leq n/2+1$, and the first question is to close the gap between the bounds. As mentioned earlier, when proving the lower bound in Theorem \ref{thmUNIQ}, our strategy is to preserve some monochromatic paths aside so that the occurrence of unused colours in the remaining part is restricted, and to control how the unused colours occur in the interface between the preserved part and the remaining part, we need to first keep a maximum collection of disjoint monochromatic claws with pairwise distinct colours, but that makes $n/4$ the best lower bound this argument can give.

Aside from Hamilton cycles, one can consider the unique Ramsey number of Hamilton paths $P_n$. The proof of Theorem \ref{thmUNIQ} can also show that $r_\text{u}(n,P_n)> n/4$. As for the upper bound, a similar construction also shows that $r_\text{u}(n,P_n)\leq n/2+1$. For instance, we can modify the construction in Proposition \ref{Upper} by taking $|V_1|=|V_2|=n/2$ and giving all edges with $e\subseteq V_1$ or $e\subseteq V_2$ the colour $n/2+1$. Therefore, we also know $n/4\leq r_\text{u}
(n,P_n)\leq n/2+1$. A natural question is to close the gap between the bounds on $r_\text{odd}
(n,P_n)$ and $r_\text{u}
(n,P_n)$. In particular, it would be interesting to see if their true values differ, and by what amount if yes.

\begin{problem}
    Let $n\geq 4$ be an integer. What are the values of $r_\mathrm{u}(n,C_n)$ and $r_\mathrm{u}(n,P_n)$?
\end{problem}
One can also consider the odd-Ramsey number of these two subgraphs, where $r_\text{odd}(n,C_n)$ is discussed when $n$ is even and $r_\text{odd}(n,P_n)$ is discussed when $n$ is odd. In \cite{boyadzhiyska2025oddramseynumbershamiltoncycles} it was shown that $r_\text{odd}(n,C_n)=\Theta(\sqrt n)$ for even $n$, and their proof of the lower bound extends to $r_\text{odd}(n,P_n)=\Omega(\sqrt n)$ for odd $n$. However, their construction for the upper bound does not extend to $r_\text{odd}(n,P_n)$, and the true magnitude of $r_\text{odd}(n,P_n)$ remains open.
\begin{problem}
    Let $n\geq 4$ be an integer. What is the magnitude of $r_\mathrm{odd}(n,P_n)$?
\end{problem}

\section*{Acknowledgements} We are very grateful to Simona Boyadzhiyska, Thomas Lesgourgues, and Kalina Petrova for many helpful conversations, and in particular for sharing the colouring for the upper bound on the unique Ramsey number of Hamilton cycles (Proposition~\ref{Upper}).

\bibliographystyle{plain}
\bibliography{Ref}

\end{document}